\documentclass[12pt]{article}
\usepackage[dvips]{epsfig}
\usepackage{amscd}
\usepackage{amssymb}
\usepackage{amsthm}
\usepackage{amsmath}
\usepackage{latexsym}

\theoremstyle{plain}
\newtheorem{thm}{Theorem}[section]
\theoremstyle{plain}
\newtheorem{lemma}[thm]{Lemma}
\newtheorem{prop}[thm]{Proposition}

\theoremstyle{definition}

\newtheorem*{remark}{Remark}

{%
\setcounter{enumi}{0}

\begin{enumerate}}%
{\end{enumerate} }
{%
\setcounter{enumi}{0}

\begin{enumerate}}%
{\end{enumerate} }

\newcommand{\B}{\ensuremath{\mathcal{B}}}

\newcommand{\M}{\ensuremath{\mathcal{M}}}

\newcommand{\eps}{\ensuremath{\varepsilon}}
\newcommand{\R}{\ensuremath{\mathbb{R}}}

\newcommand{\be} {\begin{equation}}
\newcommand{\ee} {\end{equation}}
\newcommand{\bea} {\begin{eqnarray}}
\newcommand{\eea} {\end{eqnarray}}
\newcommand{\Bea} {\begin{eqnarray*}}
\newcommand{\Eea} {\end{eqnarray*}}

\newcommand{\goto}{\ensuremath{\rightarrow}}
\newcommand{\grad}{\ensuremath{\nabla}}
\newcommand{\De} {\Delta}
\newcommand{\cpn} {\mathbb{CP}^{N}}
\newcommand{\pl} {\partial}

\catcode`\@=11

\numberwithin{equation}{section} \allowdisplaybreaks
\begin{document}
%
%

\title{On the solutions of a singular elliptic equation concentrating on a circle }

\author{ B. B. Manna\footnote{ B. B. Manna, TIFR CAM , Bangalore, email: bhakti@math.tifrbng.res.in} ,
P. N. Srikanth \footnote{P. N. Srikanth ,TIFR CAM , Bangalore, email: srikanth@math.tifrbng.res.in}}


\date{\today}
\maketitle
\begin{abstract} Let $A=\{x\in \R^{2N+2} : 0< a< |x| <b\}$ be an annulus. Consider the following singularly perturbed elliptic problem on $A$
 \begin{equation}
  \begin{array}{lll}
  -\eps^2{\De u} + |x|^{\alpha}u = |x|^{\alpha}u^p,  &\mbox{ \qquad in } A \notag\\
  u>0   &\mbox{ \qquad in }  A \\
  \frac{\partial u}{\partial\nu} = 0  &\mbox{   \qquad on } \partial A\\
  \end{array}
 \end{equation}
$1<p<2^*-1$. 
We shall show that there exists a positive solution $u_\eps$ concentrating on an $S^1$  orbit as $\eps\to 0$. We prove this by reducing the problem 
to a lower dimensional one and analyzing a single point concentrating solution in the lower dimensional space. We make precise how the single 
peak concentration depends on the parameter $\alpha$. 
\end{abstract}

\section{Introduction}
Consider the following singularly perturbed elliptic equation with super linear nonlinearity on an annulus in $\R^{2N}$
 \begin{equation}
  \begin{array}{lll}
  -\eps^2{\De u} + |x|^{\alpha}u = |x|^{\alpha}u^p,  &\mbox{ \qquad in } A \label{E0.1}\\
  u>0   &\mbox{ \qquad in }  A \\
  \frac{\partial u}{\partial\nu} = 0  &\mbox{   \qquad on } \partial A\\
  \end{array}
 \end{equation}
$1<p<2^*-1$, $\eps$ is a singular perturbation parameter. $A=\{x\in \R^{2N+2} : 0< a< |x| <b\}$. $\alpha$ any real number.  

The result of point concentration on bounded domains has been well established by several authors \cite{MR1115095, MR1219814, MR1639546}. In these 
works, the behavior of the least energy solutions and there concentration phenomena has been studied. For the Dirichlet problem W.-M. Ni and J. Wei
\cite{MR1639546}, have shown that the least energy solution can have at most one local maximum and the point of maximum converges to a point 
which stays at maximum distance from the boundary. W.-M. Ni and Takagi in \cite{MR1115095, MR1219814} have analyzed the Neumann problem, where they
have also shown that a least energy solution can have at most one local maximum but it will lie on the boundary for sufficiently small $\eps$ and
it will converge to a point of maximum mean curvature of the boundary. Later J. Byeon and J. Park in \cite{MR2180862} have generalized the same 
results for both boundary conditions on a Riemannian manifold. 

Also the $N-1$ dimensional (sphere) concentration of the problem in the presence of a potential has been studied by A. Ambrosetti, A. Malchiodi 
and W.-M. Ni in \cite{MR2056434} where they have looked at the radial solutions and established the concentration phenomena which
depends upon the behavior of the potential. 

The inspiration for the work comes from the result by Bernhard Ruf and P. N. Srikanth \cite{MR2608946} where the authors have found a 
solution concentrating on a circle in the case of Dirichlet data. The problem was considered in dimension 4 and using the $S^1$ action on $S^3$ the problem is reduced to another 
Singularly perturbed elliptic problem on an annulus in dimension 3. In a recent work, \cite{arXiv:1210.0782} Pacella and Srikanth have generalized this 
result to find solutions concentrating on $S^{N-1}$ orbit where the domain is an annulus in $\R^{2N}$.

We adapt ideas from \cite{MR2180862} and \cite{MR2608946} in the present case where the reduced problem is studied on a warped-product manifold.
The main theorem we prove here is the following:

\begin{thm}\label{Th0.1}
 There is a solution of (\ref{E0.1}) concentrating on an $S^1$ orbit, which lies on the inner boundary for $\alpha<\frac{2}{2N-1}$, on the outer 
 boundary for $\alpha\ge\frac{2}{2N-1}$. 
\end{thm}

We can re-write the equation (\ref{E0.1}) as 
\begin{equation}
  \begin{array}{lll}
  -\eps^2{\De u} + |x|^{\alpha}u = |x|^{\alpha}f(u),  &\mbox{ \qquad in }  A \label{E0.2}\\
  u>0   &\mbox{ \qquad in } A \\
  \frac{\partial u}{\partial\nu} = 0  &\mbox{   \qquad on } \partial A\\
  \end{array}
 \end{equation}
 where $f(t)= t^p$, for $t>0$ and $=0$, for $t\le0$. Then any solution of (\ref{E0.2}) is positive and hence a solution of (\ref{E0.1}) also.
 
 The basic idea is to reduce the problem to lower dimension using an $S^1$ action which leads to the Hopf Fibration. Recall that the annulus has 
 an warped product structure as 
 \begin{align}
  I\times S^{2N+1}\notag 
 \end{align}
with the product metric
 \begin{align}
  g_A=dr^2+r^2dS_{2N+1}\label{0.1} 
 \end{align}
. We write the 
 co-ordinates of $S^{2N+1}$ as $(z_1,z_2,...,z_{n+1})$. The the Hopf map $S^1\hookrightarrow S^{2n+1}\rightarrow \mathbb{CP}^{n}$ can be described 
 as $(z_1,z_2,...,z_{n+1}) \rightarrow (\frac{z_1}{z_i},\frac{z_2}{z_i},,...,1,...,\frac{z_{n+1}}{z_i})$ provided $z_i\neq 0$. Also under this 
 transformation $\De_{S^{2N+1}}$ goes to $\De_{\cpn}$ (Details can be found in \cite{MR2250232}). Also choosing a proper scaling of the radius we 
 reduce the problem in a lower dimensional singularly perturbed problem on the warped product manifold $\M=I'\times_f \cpn$ with the product metric 
 $g_\M=ds^2+ \frac{2N-1}{2N}r^2g_{\cpn}^2$. Where $I'=(\frac{2N}{2N-1})^{\frac{1}{2N-1}}(a^{\frac{2N}{2N-1}}, b^{\frac{2N}{2N-1}})$, and $g_{\cpn}$ 
 is the Fubini Study metric on $g_{\cpn}$(for details please look at the appendix).
 
 We shall seek for a solution of the reduced equation and try to get the behavior of the sequence of solutions. We shall prove that as $\eps\goto 0$
 'up-to a subsequence' the solutions concentrates at a single point on the boundary. We lift the solution in the annulus to get solutions 
 concentrating on $S^1$.
 
 \section{The group action and reduction}
 Let $A=\{x\in\R^{2N+2}:0<a<|x|<b\}$ is an annular domain in $\R^{2N+2}$. We can express $A$ as a product manifold $A=I\times_r S^{2N+1}$, where 
 $I=(a,b)$, with the product metric 
 \begin{align}\label{2.1}
 g=dr^2+r^2dS_{2N+1}
 \end{align}Let $H_{0,rad}^1\subset H_0^1(A)$ where $H_{rad}^1$ denote the space 
of radial functions in $H^1(A)$ consists of radial functions.

Consider a suitable co-ordinate representation of The annulus $A$ such that any point $z\in A$ can be written as 
\begin{align}
z\equiv z(r,t_1,...,t_n,\theta_1,...,\theta_{N+1})\notag
\end{align}
where $a<r<b$ and $0\le t_i<\pi/2, (i=1,2,...,n)$ and $0\le \theta_j<2\pi, (j=1,2,...,n+1)$.

 Note that the angles $0\le \theta_i<2\pi(i=1,2)$ represents the angle between $(x_{2i-1},x_{2i})$ in the $x_{2i-1},x_{2i}$, and $0\le t_j<
 \pi/2$ is the angle between the respective planes.
 
 Now consider the following one parameter group action $T_\tau$ on $A$ : Define 
 \begin{align}
  z(r,t_1,...,t_n,\theta_1,...,\theta_{N+1})=(x_1,x_2,...,x_{2N+2})\notag
 \end{align}
Then let $T_\tau(z)=z(r,t_1,...,t_n,\theta_1+\tau,...,\theta_{N+1}+\tau)$ for $\tau\in[0,2\pi)$. Define $H_{\sharp}^1(A)\subset H^1(A)$ by
 \begin{align}
  H_{\sharp}^1(A)=\{u\in H^1(A) : u(T_\tau(z))=u(z), \forall \tau \in [0,2\pi)\}.
 \end{align}
 \begin{lemma}\label{L2.1}
  $T_\tau : A \rightarrow A$ is a fixed point free group action.
 \end{lemma}
 
 \begin{remark}
  The above lemma is important. As any solution concentrating on a fixed point shall not give any concentrating orbit.
 \end{remark}
For $u\in H_{\sharp}^1(A)$ we see that $u(T_\tau(z))=u(z), \forall \tau \in [0,2\pi)$, so $u(T_{-\theta_{N+1}}(z))=u(z)$ Let us define new variables
$\psi_i=\theta_i-\theta_{N+1}$ and define $v(r,t_1,...,t_n,\psi_1,...,\psi_{N})=u(T_{-\theta_{N+1}}(z))$

 \begin{lemma}{\label{L2.2}}
  $v$ is well defined.
 \end{lemma}
 The proofs of lemma \ref{L2.1} and lemma \ref{L2.2} are given in the appendix. 
 
 Note that any solution of (\ref{E0.2}) is a critical point of the 
 functional 
 \begin{align}
  J_\eps(u)&=\int_{A}\Big(\frac{\eps^2}{2}|\grad u|^2 + |x|^\alpha\frac{u^2}{2}-|x|^\alpha F(u)\Big)dx\notag\\
  &=\int_{I\times S^{2N+1}}\Big[\frac{\eps^2}{2}(u_r^2+\frac{1}{r^2}|\grad_{S^{2N+1}}u|^2)+r^\alpha\frac{u^2}{2}-r^\alpha F(u)\Big]r^{2N+1}drd\sigma_{S^{2N+1}}\notag
 \end{align} 
 
 Let $\Psi : A\to\mathcal{M}$ be the Riemann submersion, under the metric (\ref{Ap A.3}), given by \be \Psi (r,\theta)=(s,\psi)\label{a2.10}\ee 
 where $\theta$ is a point on $S^{2N+1}$ and $\psi$ is the image of $\theta$ of the map (\ref{Ap A.6}),(see appendix A).  Denote the projections 
 by $s, \sigma$ from $\M$ onto $I'$ and $\cpn$ respectively. Then under this change of co-ordinates the energy functional $J_\eps(u)$ takes the 
 form (see appendix B)
 \be 2\pi\int_{I'\times\cpn}\Big[\frac{\eps^2}{2}(v_s^2+(\frac{2N}{2N-1}s)^{-2}|\grad_{\cpn}v|^2)\notag\ee
  \be +\Big(\frac{2N-1}{2N}\Big)^{\frac{4N+2+\alpha}{2N}}\Big(\frac{v^2}{2s^{\frac{\alpha+2-2N\alpha}{2N}}}-\frac{F(v)}{s^{\frac{\alpha+2
  -2N\alpha}{2N}}}\Big)\Big] s^{2N}dsdV_{\cpn}\label{2.3}\ee
 
 So the critical points of $J_\eps$ in $H_\sharp^1(A)$ corresponding to solutions of the following equation
\begin{equation}
 \begin{array}{lll}
  -\eps^2{\De_g v} + (\frac{2N-1}{2N})^{\frac{4N+2+\alpha}{2N}}\Big(\frac{v}{|s(p)|^\eta}-\frac{f(v)}{|s(p)|^\eta}\Big)=0  &\mbox{ \qquad in } \M \notag\\
  v>0   &\mbox{ \qquad in }  \M \\
  \frac{\partial v}{\partial\nu} = 0  &\mbox{   \qquad on } \partial \M\\
  \end{array}
\end{equation}
where $\eta=\frac{\alpha+2-2N\alpha}{2N}$ and $s$ is the projections from $\M$ onto $I'$.

Or equivalently replacing $\eps$ by $(\frac{2N-1}{2N})^{-\frac{4N+2+\alpha}{4N}}\eps$ we have
\begin{equation}
 \begin{array}{lll}
  -\eps^2{\De_g v} + \frac{v}{|s(p)|^\eta}-\frac{f(v)}{|s(p)|^\eta}=0  &\mbox{ \qquad in } \M \label{E2.1}\\
  v>0   &\mbox{ \qquad in }  \M \\
  \frac{\partial v}{\partial\nu} = 0  &\mbox{   \qquad on } \partial \M\\
  \end{array}
\end{equation} 
where $f(t)=0 \mbox{ for } t<0 \mbox{, } t^p \mbox{ for } t\ge0$.

 We look for a single peak solution for the equation (\ref{E2.1}). In order to do this We will find a Mountain Pass solution to equation (\ref{E2.1}) 
 and analyze its behavior to prove our theorem. We shall first analyze
 the limit equation of (\ref{E2.1}). It turns out that (as we shall see in the next section), for a mountain pass solution $u_\eps$ of (\ref{E2.1}) 
 the transformed function $U_\eps(x)=u_\eps(\eps x_{P_0}^{-1}(x))$ converges uniformly to a solution of the equation (in $C^2_{loc}$ sense).
 \begin{equation}
 \begin{array}{lll}
  {\De U} - \frac{U}{|s(P_0)|^\eta}-\frac{f(U)}{|s(P_0)|^\eta}=0, U>0  &\mbox{ \qquad in } \{x\in\R^{2N+1} : \mu\cdotp x>0\} \\
  \frac{\partial U}{\partial\mu} = 0  &\mbox{   \qquad on } \{x\in\R^{2N+1} : \mu\cdotp x=0\}\label{E2.2}\\
  \lim_{|x|\goto\infty}U(x)=0
  \end{array}
\end{equation} 
 for some unit vector $\mu\in\R^{2N+1}, \ P_0\in\partial\M$ and for some local co-ordinate $x_{P_0}$ around $P_0$. Let $|s(P_0)|^\eta=\kappa$.

  Let $V(x)=U(\frac{x}{\sqrt\kappa})$, the for $U$ solving (\ref{E2.2}), $V$ satisfies
 \begin{equation}
 \begin{array}{lll}
  {\De V} - V + f(V)=0, V^2>0  &\mbox{ \qquad in } \{x\in\R^{2N+1} : \mu\cdotp x>0\} \\
  \frac{\partial V}{\partial\mu} = 0  &\mbox{   \qquad on } \{x\in\R^{2N+1} : \mu\cdotp x=0\}\label{E2.3}\\
  \lim_{|x|\goto\infty}V(x)=0
  \end{array}
\end{equation}

 With choosing proper co-ordinate chart around $P_0$ we can have $\mu=(0,...0,1)$. For $U\in H^1(\R^{2N+1}_+)$ define
 \begin{align}
  \Gamma(U)=\int_{\R^{2N+1}_+}\Big[\frac{1}{2}|\grad U|^2+\frac{U^2}{2\kappa}-\frac{F(U)}{\kappa}\Big]dx\label{a2.4}
 \end{align}

 Let $\mathcal{B}$ be the set of all solutions $U$ of equation (\ref{E2.2}) with $\mu=(0,0,...,0,1)$ satisfying
 \be U(0)=\max_{x\in \R^{2N+1}_+} U(x)\notag \ee
 Then the following results are well known about $U$
 \begin{prop}
  Any $V\in \mathcal{B}$ is radially symmetric and $V'(r)<0$ for $r>0$. Moreover there exist a $C,c>0$ such that $V(x)+|\grad V(x)|\le C\exp{-c|x|}$.
  Also the set $\mathcal{B}$ is compact in $H^1(\R^{2N+1}_+)$.  
 \end{prop}
 \begin{prop}
  For $U\in\mathcal{B}$ we have for $1\le j \le 2N, \ m\ge 0$ the following identities
  \begin{align}
   &(i) \ \int_{\R^{2N+1}_+}[\frac{1}{2}|\grad U|^2+\frac{U^2}{2\kappa}-\frac{F(U)}{\kappa}]x_{2N+1}^mdx=\frac{m+1}{2N+m+1}\int_{\R^{2N+1}_+}|\grad U|^2x_{2N+1}^mdx\label{a2.5}\\
   &(ii) \ \int_{\R^{2N+1}_+}\Big(\frac{\partial U}{\partial x_j}\Big)^2 x_{2N+1}^mdx=\frac{1}{2N+m+1}\int_{\R^{2N+1}_+}|\grad U|^2x_{2N+1}^mdx\label{a2.6}\\
   &(iii) \ \int_{\R^{2N+1}_+}\Big(\frac{\partial U}{\partial x_{2N+1}}\Big)^2 x_{2N+1}^mdx=\frac{m+1}{2N+m+1}\int_{\R^{2N+1}_+}|\grad U|^2x_{2N+1}^mdx\label{a2.7}
  \end{align}
 and the Pohozaev identity
  \begin{align}
   &(iv) \ \int_{\R^{2N+1}_+}\Big[\frac{2N-1}{2}|\grad U|^2-(2N+1)\frac{U^2}{2\kappa}-(2N+1)\frac{F(U)}{\kappa}\Big]dx=0\label{a2.8}
  \end{align}

 \end{prop}
 
 For $u\in C^\infty{\M}$ define
 \begin{align}
  ||u||_\eps=\int_\M\Big[\eps^2|\grad_g u|^2 + \frac{u^2}{|s(p)|^\eta}\Big]dv_g
 \end{align}
 We can easily verify that $||.||_\eps$ defines an equivalent norm on $C^\infty(\M)$ as the usual $H^1$ norm on $\M$. let $H_\eps(\M)$ be the 
 completion of $C^\infty(\M)$ in the norm $||.||_\eps$. The for $u \in H_\eps(\M)$ we have 
 \begin{align}
  \Gamma_\eps(u)=\frac{1}{2}||u||_\eps^2-\int_\M \frac{F(u)}{|s(p)|^\eta}dv_g\label{a2.9}
 \end{align}
 
 \section{Some geometric preliminaries}
For $P_0\in\partial\M$, let $(x_1, x_2,...,x_{2N})$ be a Riemann normal coordinates on $\partial\M$ at $P_0$. For a point $q$ close enough to $P_0$, 
let $x_{2N+1}$ be the distance of $q$ from $\partial\M$. The chart $x_{P_0}^\pl = (x_1,...,x_{2N},x_{2N+1})$ is known as Fermi co-ordinate at 
$P_0$. In these co-ordinates the arc length $dl^2$ can be written as :
\begin{align}
 dl^2=dx^2_{2N+1}+\sum_{i,j=1}^{2N}g_{ij}(x',x_{2N+1})dx_idx_j\label{a4.1}
\end{align}
where $g_{ij}=g(\frac{\partial}{\partial x_i}, \frac{\partial}{\partial x_j})|_{(x',x_{2N+1})}$.
\begin{lemma}\label{L4.1}
 For $p$ close enough to $P_0$ we have 
 \begin{align}
  x_{2N+1}(p) = |s(p)-s(P_0)|\label{a4.2}
 \end{align}

\end{lemma}
\begin{proof}
 We have $\M=B\times_f F$, where $B=I'$ and $F=\cpn$ and the metric $g=ds^2+(\frac{2N}{2N-1}r)^2g_{\cpn}$. Now $x_{2N+1}(p)=dist_g(p,\pl\M)$. Let
 $\sigma(p)=p'$ and take the point $(s(P_0),p')\in\pl\M$. Consider the path $\gamma(t)=(s(p)+t(s(P_0)-s(P),p'), \ t\in[0,1]$ joining $p$ and
 $(s(P_0),p')\in\pl\M$. Then
 \begin{align}
  x_{2N+1}(p)\le dist_g(p,(s(P_0),p'))\le\int_0^1(g(\gamma'(t),\gamma'(t)))^{1/2}dt\notag 
 \end{align}
Now $\gamma'(t)=(s(P_0)-s(p))\frac{\pl}{\pl r}$ implies $g(\gamma'(t),\gamma'(t))=|s(P_0)-s(p)|^2$ So we have
 \begin{align}
  x_{2N+1}(p)\le \int_0^1|s(P_0)-s(p)|dt=|s(P_0)-s(p)|\notag 
 \end{align}
 Now from the compactness of $\cpn$ we have a $\tilde{p}\in\pl\M$ such that 
 \begin{align}
  x_{2N+1}(p)=dist_g(p,\pl\M)=dist_g(p,\tilde{p})\notag
 \end{align}
Let $\eta(t)$ be a geodesic joining $p$ and $\tilde{p}$ which is length minimizing. Then
\begin{align}
 x_{2N+1}(p)&=\int_0^1(g(\eta'(t),\eta'(t)))^{1/2}dt\notag\\
 &=\int_0^1((\frac{dr}{dt})^2+\mbox{ some positive terms})^{1/2}dt\notag\\
 &\ge |\int_0^1 dr|=|s(1)-s(0)|=|s(p)-s(P_0)|\notag 
\end{align}

\end{proof}
Now for $P_0\in\pl\M$, let $\mathcal{P}(P_0)$ be the projection of $T_{P_0}(\M)$ onto $T_{P_0}(\pl\M)$. The second fundamental form $\Pi(X,Y)$ is 
defined as $\Pi(X,Y)=\nabla_XY-\mathcal(\nabla_XY)$ for $X,Y\in T(\pl\M)$. The mean curvature of $\pl\M$ at $P_0\in\pl\M$ is defined as the trace 
of $\Pi$ at $P_0\in\pl\M$. Let $(X_1,...,X_{2N})$ be an orthogonal vector field in a neighborhood of $P_0$ in $\pl\M$. It is well known that the 
second fundamental form 
\begin{align}
 \Pi(X_i,X_j)(P_0)=-\frac{\grad f}{f}|_{P_0}=-\frac{1}{s(P_0)}\label{a4.3}
\end{align}

Let $g_{ij}=g(\frac{\partial}{\partial x_i}, \frac{\partial}{\partial x_j})$, corresponding to the Fermi co-ordinate $x_{P_0}^\pl$ at $P_0$.
Let $g^{kl}=((g_{ij}))^{-1}_{kl}$ and $|g|=det((g_{ij}))$. Then it is well known that
 \begin{align}
  g^{ij}(x)=\delta_{ij}+2h_{ij}x_{2N+1}+O(|x|^2)\label{a4.4}\\
  \sqrt(|g|)=1-2NHx_{2N+1}+O(|x|^2)\label{a4.5}
 \end{align}
for $x$ small enough, $(h_{ij})_{1\le i,j\le 2N}$ is the second fundamental form and $H(p)=-\frac{1}{s(P)}$ is the mean curvature at $P\in\pl\M$.

Define the functional $L : \pl\M\times\mathcal{B}\goto\R$ by
\begin{align}
 H(p,U)&=-\sum_{i,j=1}^{2N}h_{ij}(p)\int_{\R^{2N+1}_+}\pl_i U\pl_j U x_{2N+1}dx \notag\\
 &+ 2NH(p)\int_{\R^{2N+1}_+}\Big(\frac{1}{2}|\grad U|^2+\frac{U^2}{2|s(p)|^\eta}-\frac{F(U)}{|s(p)|^\eta}\Big)x_{2N+1}dx\label{a4.6}
\end{align}
 Then we have the following
 \begin{prop}\label{P4.1}
  Let $V\in\mathcal{B}$ be a radially symmetric solution of (\ref{E2.2}) with $\mu=(0,...,0,1)$. Then for any $p\in\pl\M$ we have 
  \begin{align}
   H(p,V)=\frac{N}{N+1}H(p)\int_{\R^{2N+1}_+}|\grad V|^2x_{2N+1}dx\label{a4.7}
  \end{align}

 \end{prop}

 In lemma (\ref{L4.1}) we have shown that for co-ordinate $x_{P_0}^\pl$, $x_{2N+1}(p) = |s(p)-s(P_0)|$. Now the boundary of $\M$ is the two 
 disjoint copies of $\cpn$. let us denote the component of the boundary corresponding to $s=(\frac{2N}{2N-1})^{\frac{1}{2N-1}}a^{\frac{2N}{2N-1}}$ 
 as $\partial\M_a$ and the component of the boundary corresponding to $s=(\frac{2N}{2N-1})^{\frac{1}{2N-1}}b^{\frac{2N}{2N-1}}$ as $\partial\M_b$. 
 Then for $q$ near $\partial\M_a$ from lemma (\ref{L4.1})  we have
  \begin{align}
   |s(q)|^{-\eta}=|s(P_0)|^{-\eta}-\eta|s(P_0)|^{-\eta-1}x_{2N+1}+O(|x_{2N+1}|^2)\label{a4.8}
  \end{align} 
 and for $q$ near $\partial\M_b$ we have
  \begin{align}
   |s(q)|^{-\eta}=|s(P_0)|^{-\eta}+\eta|s(P_0)|^{-\eta-1}x_{2N+1}+O(|x_{2N+1}|^2)\label{a4.9}
  \end{align}

\section{The MP solution and proof of Theorem(\ref{Th0.1})}
Here we shall work with the Fermi coordinate as we have discussed earlier around a point $P_0$ on the boundary of the manifold. We denote it by 
$x_{P_0}^\pl$. Let $\delta$ be small enough such that $x_{P_0}^\pl$ is a diffeomorphism from $\{x\in\M \backslash \pl\M: dist_g(P_0,x)<\delta$ \}to 
an open neighborhood of $0$ in $\R^{2N+1}_+$. Note that $x_{P_0}^\pl$ maps $\pl\M$ into $\R^{2N}=\pl\R^{2N+1}_+$ locally around $P_0$. Define $\phi_\gamma \in
C_0^\infty(\R^{2N+1},[0,1])$ as 

\[\phi_\gamma(x) = 
     \begin{cases}
      1&\text{if $ |x| \le \gamma $}\notag \\
      0&\text{if $ |x| \ge 2\gamma $}\notag \\
     \end{cases}
\]
For $V\in\mathcal{B}$ define $Z_{\eps,t}^\gamma(p)=\phi_\gamma(\frac{x_{P_0}^\pl(p)}{t})V(\frac{x_{P_0}^\pl(p)}{\eps t})$. Let us consider that
$B_+(0,2t\gamma)\subset x_{P_0}^\pl(B_g(P_0,\delta))$. Then 
\begin{align}
 \Gamma_\eps(Z_{\eps,t}^\gamma)=\int_{\M}\Big(\frac{\eps^2}{2}|\grad_g Z_{\eps,t}^\gamma|^2 + \frac{|Z_{\eps,t}^\gamma|^2}{2|s(p)|^\eta}- \frac{F(Z_{\eps,t}^\gamma)}{|s(p)|^\eta}\Big)dv_g\label{a5.1}
\end{align}
 to get the Mountain pass solution we need to first simplify the terms of the above expressions
 \begin{align}
  &\int_{\M}\frac{\eps^2}{2}|\grad_g Z_{\eps,t}^\gamma|^2dv_g\notag\\
  =& \int_{B_g(P_0,\delta)}\frac{\eps^2}{2}|\grad_g Z_{\eps,t}^\gamma|^2dv_g\notag\\
  =& \frac{\eps^2}{2}\int_{B_+(0,2t\gamma)}\sum_{i,j=1}^{2N+1}g^{ij}(x)\pl_i(\phi_\gamma(\frac{x}{t})V(\frac{x}{\eps t}))\pl_j(\phi_\gamma(\frac{x}{t})V(\frac{x}{\eps t})\sqrt(|g|)dx\notag\\
  =& \frac{1}{2}\eps^{2N+1}t^{2N-1}\int_{B_+(0,\frac{2\gamma}{\eps})}\sum_{i,j=1}^{2N+1}g^{ij}(\eps ty)\pl_i(\phi_\gamma(\eps y)V(y))\pl_j(\phi_\gamma(\eps y)V(y))\sqrt(|g|(\eps ty))dy\notag
 \end{align}
Where $x=\eps ty$. From the expressions of $|g|$ and $g^{ij}$ in (\ref{a4.4}) and (\ref{a4.5}) we get
\begin{align}
 &\sum_{i,j=1}^{2N+1}g^{ij}(\eps ty)\pl_i(\phi_\gamma(\eps y)V(y))\pl_j(\phi_\gamma(\eps y)V(y))\sqrt(|g|(\eps ty))\notag\\
 =& |\grad (\phi_\gamma(\eps y)V(y))|^2-2NH(P_0)\eps t|\grad (\phi_\gamma(\eps y)V(y))|^2y_{2N+1}\notag\\
 &+2\eps t \sum_{i,j=1}^{2N}h_{ij}(P_0)\pl_i(\phi_\gamma(\eps y)V(y))\pl_j(\phi_\gamma(\eps y)V(y))y_{2N+1}\notag
  \end{align}
 \begin{align}
 &+\sum_{i,j=1}^{2N}O(|\eps ty|^2)\pl_i(\phi_\gamma(\eps y)V(y))\pl_j(\phi_\gamma(\eps y)V(y))\notag
\end{align}

 So we get 
 \begin{align}
  &\int_{\M}\frac{\eps^2}{2}|\grad_g Z_{\eps,t}^\gamma|^2dv_g\notag\\
  =&\eps^{2N+1}\Big[\int_{B_+(0,\frac{2\gamma}{\eps})}\frac{t^{2N-1}}{2}|\grad (\phi_\gamma(\eps y)V(y))|^2dy\notag\\
  &-2NH(P_0)\frac{\eps t^{2N}}{2}\int_{B_+(0,\frac{2\gamma}{\eps})}|\grad (\phi_\gamma(\eps y)V(y))|^2y_{2N+1}dy\notag\\
  &+\eps t^{2N}\sum_{i,j=1}^{2N}h_{ij}(P_0)\int_{B_+(0,\frac{2\gamma}{\eps})}\pl_i(\phi_\gamma(\eps y)V(y))\pl_j(\phi_\gamma(\eps y)V(y))y_{2N+1}dy\Big]\notag\\
  &+t^{2N+1}O(\eps^{2N+3})\label{a5.2}
 \end{align}
We got the above inequality using the decay estimate of $V$ and $DV$ near infinity. Again using the same decay estimate we can easily show the 
following estimates
\begin{align}
 &\int_{B_+(0,\frac{2\gamma}{\eps}}|\grad (\phi_\gamma(\eps y)V(y))|^2dy\notag\\
 =&\int_{B_+(0,\frac{\gamma}{\eps})}|\grad V(y)|^2dy + \int_{B_+(0,\frac{2\gamma}{\eps})\backslash B_+(0,\frac{\gamma}{\eps})}|\grad (\phi_\gamma(\eps y)V(y))|^2dy\notag\\
 =&\int_{B_+(0,\frac{\gamma}{\eps})}|\grad V(y)|^2dy + O(\eps^2)\notag\\
 =&\int_{\R^{2N+1}_+}|\grad V(y)|^2dy + O(\eps^2)\notag
\end{align}
 Estimating all the terms in the same way we get the expression for the first integral as
 \begin{align}
  &\frac{\eps^2}{2}\int_{\M}|\grad_g Z_{\eps,t}^\gamma|^2dv_g\notag\\
  =&\eps^{2N+1}\Big[\frac{t^{2N-1}}{2}\int_{\R^{2N+1}_+}|\grad V|^2dy
  -2NH(P_0)\frac{\eps t^{2N}}{2}\int_{\R^{2N+1}_+}|\grad V|^2y_{2N+1}dy\notag\\
  &+\eps t^{2N}\sum_{i,j=1}^{2N}h_{ij}(P_0)\int_{\R^{2N+1}_+}\pl_iV\pl_jVy_{2N+1}dy + t^{2N+1}O(\eps^2)\Big]\label{a5.3}
 \end{align}
 The second term is
 \begin{align}
  &\int_{\M}\frac{|Z_{\eps,t}^\gamma|^2}{2|s(p)|^\eta}dv_g\notag\\
  &=\frac{1}{2}\int_{B_g(P_0,\delta)}\frac{|Z_{\eps,t}^\gamma|^2}{|s(p)|^\eta}dv_g\notag\\
  &=\frac{1}{2}\int_{B_+(0,2t\gamma)}\frac{1}{|s(p)|^\eta}\left(\phi_\gamma(\frac{x}{t})V(\frac{x}{\eps t})\right)^2\sqrt{|g(x)|}dx\notag\\
  &=\frac{1}{2}\int_{B_+(0,\frac{2\gamma}{\eps})}\eps^{2N+1}t^{2N+1}\frac{1}{2|s(x_{p_0}^\delta)^{-1}(\eps ty)|^\eta}\left(\phi_\gamma(\eps y)V(y)\right)^2\sqrt{|g(\eps ty)|}dy\notag
 \end{align}
Here we have make the change of variable $x=\eps ty$ as before. To simplify the above expression let us expand $s(p)$ around $s(P_0)$. W.L.O.G we 
can take $P_0$ on the inner boundary. The same approach shall work for $P_0$ on the outer boundary.
\begin{align}
 &|s(p)|^{-\eta}\sqrt{|g|(\eps ty)}\notag\\
 =&(|s(P_0)|^{-\eta}-\eps t\eta|s(P_0)|^{-\eta-1}y_{2n+1}+O(|\eps ty_{2n+1}|^2))(1-2NH(P_0)\eps ty_{2n+1}+O(|\eps ty_{2n+1}|^2)\notag\\
 =& |s(P_0)|^{-\eta}-\eps t\eta|s(P_0)|^{-\eta-1}y_{2n+1}-2NH(P_0)\eps ty_{2n+1}|s(P_0)|^{-\eta} +O(|\eps ty|^2
\end{align}
 Then we have 
 \begin{align}
  &\int_{B_+(0,\frac{2\gamma}{\eps})}\frac{\eps^{2N+1}t^{2N+1}}{2|s(p)|^{\eta}}(\phi_\gamma(\eps y)V(y))^2\sqrt{|g(\eps ty)|}dy\notag\\
  =&\eps^{2N+1}\Big[\int_{B_+(0,\frac{2\gamma}{\eps})}\Big(\frac{t^{2N+1}}{2}\frac{(\phi_\gamma(\eps y)V(y))^2}{|s(P_0)|^\eta}-
  2NH(P_0)\eps t\frac{(\phi_\gamma(\eps y)V(y))^2}{|s(P_0)|^\eta}y_{2N+1}\notag\\
  &-\eps t\frac{(\phi_\gamma(\eps y)V(y))^2}{|s(P_0)|^{\eta+1}}y_{2N+1}\Big)dy+O(\eps^2)\Big]\notag\\
  =&\eps^{2N+1}\Big[t^{2N+1}\int_{\R^{2N+1}_+}\frac{V^2(y)}{2|s(P_0)|^\eta}dy
  -2NH(P_0)\eps t^{2N+2}\int_{\R^{2N+1}_+}\frac{V^2(y)}{2|s(P_0)|^\eta}y_{2N+1}dy\notag\\
  &-\eps t^{2N+2}\int_{\R^{2N+1}_+}\frac{V^2(y)}{2|s(P_0)|^{\eta+1}}y_{2N+1}dy+O(\eps^2)\Big]\label{a5.5}
 \end{align}
 Similarly we have
 \begin{align}
  &\int_{\M}\frac{F(Z_{\eps,t}^\gamma)}{|s(p)|^\eta}dv_g\notag\\
  =&\eps^{2N+1}\Big[t^{2N+1}\int_{\R^{2N+1}_+}\frac{F(V)}{|s(P_0)|^\eta}dy
  -2NH(P_0)\eps t^{2N+2}\int_{\R^{2N+1}_+}\frac{F(V)}{|s(P_0)|^\eta}y_{2N+1}dy\notag\\
  &-\eps t^{2N+2}\int_{\R^{2N+1}_+}\frac{F(V)}{|s(P_0)|^{\eta+1}}y_{2N+1}dy+O(\eps^2)\Big]\label{a5.6}
 \end{align}

 Finally combining (\ref{a5.3}), (\ref{a5.5}), (\ref{a5.6}) we get
 \begin{align}
  \eps^{-(2N+1)}\Gamma_\eps(Z_{\eps,t}^\gamma)= I_1(t)-\eps I_2(t)-\eps I_3(t)+O(\eps^2)\label{a5.7}
 \end{align}
where
 \begin{align}
  I_1(t)=&\int_{\R^{2N+1}_+}\Big(\frac{t^{2N-1}}{2}|\grad V|^2dy+\frac{t^{2N+1}}{2}\frac{V^2(y)}{|s(P_0)|^\eta}-t^{2N+1}\frac{F(V)}{|s(P_0)|^\eta}\Big)dy\notag\\
  I_2(t)=&t^{2N+2}\int_{\R^{2N+1}_+}\Big(\frac{V^2(y)}{2|s(P_0)|^{\eta+1}}-\frac{F(V)}{|s(P_0)|^{\eta+1}}\Big)y_{2N+1}dy\notag\\
  I_3(t)=&2NH(P_0)\int_{\R^{2N+1}_+}\Big(t^{2N}\frac{1}{2}|\grad V|^2+t^{2N+2}\frac{V^2(y)}{2|s(P_0)|^{\eta}}-t^{2N+2}\frac{F(V)}{|s(P_0)|^\eta}\Big)y_{2N+1}dy\notag\\
  &+t^{2N}\sum_{i,j=1}^{2N}h_{ij}(P_0)\int_{\R^{2N+1}_+}\pl_iV\pl_jVy_{2N+1}dy\notag
 \end{align}
Now note that $\Gamma_\eps(Z_{\eps,t}^\gamma)=\eps^{2N+1}[I_1(t)+O(\eps)]$. uniformly for $P_0\in\pl\M$ and $t\in\R$ (here we can take any of the 
boundaries and can get the same expression up-to order of $\eps$). From the Pohozaev's identity we see that $\exists t_0$ satisfying
\begin{align}
 t^{2N-1}\int_{\R^{2N+1}_+}\frac{1}{2}|\grad V|^2dy+t^{2N+1}\int_{\R^{2N+1}_+}\Big(\frac{V^2(y)}{\kappa}dy-\frac{F(V)}{\kappa}\Big)dy<-1\notag
\end{align}
for all $t\ge t_0$. Now choose $\gamma$ small enough so that $B_+(0,2t_0\gamma)\subset x_{P_0}^\pl(B_g(P_0,\delta))$. From the compactness of the
boundary and the regularity of $Z_{\eps,t}^\gamma$ and $\Gamma_\eps$ we get the existence of $\eps_0>0, t_0>0$ such that 
$\Gamma_\eps(Z_{\eps,t}^\gamma)<0$ for all $0<\eps<\eps_0$, $t\ge t_0$ and $P_0\in\pl\M$.  Define
\be c_\eps=\inf_{\beta\in\rho} \max_{t\in [0,1]}\Gamma_\eps(\beta(t))\label{a5.8}\ee
where $\rho=\{\beta\in C([0,1], H_\eps(\M) ) : \beta(0)=0 , \beta(1)=Z_{\eps,t_0}^\gamma\}$.

 \begin{lemma}
 $c_\eps$ does not depend upon $p\in\pl\M$ and $V\in\B$
 \end{lemma}
 
 \begin{proof}
 Same proof as given in lemma 3.1 in \cite{MR1115095}.
 \end{proof}
 
  

 Let $\beta(t)=Z_{\eps,tt_0}^\gamma$, $t\in [0,1]$. then it follows that $\beta(0)=lim_{t\goto 0}\beta(t)=0$ and $\beta(1)=Z_{\eps,t_0}^\gamma$.
Moreover we see from (\ref{a5.7}) that
\begin{align}
 \Gamma_\eps(\beta(t))=&\eps^{2N+1}\Big[(tt_0)^{2N-1}\int_{\R^{2N+1}_+}\frac{1}{2}|\grad V|^2dy+
 (tt_0)^{2N+1}\int_{\R^{2N+1}_+}\frac{V^2(y)}{2|s(P_0)|^\eta}dy\notag\\
 &-(tt_0)^{2N+1}\int_{\R^{2N+1}_+}\frac{F(V)}{|s(P_0)|^\eta}dy + o(\eps)\Big]\notag
\end{align}
uniformly for $t\in [0,1]$. So 
\begin{align}
 &\Gamma_\eps(\beta(t))\le max_{t\in [0,t_0]}\Big[\frac{t^{2N-1}}{2}\int_{\R^{2N+1}_+}|\grad V|^2dy 
 +t^{2N+1}\int_{\R^{2N+1}_+}\Big(\frac{V^2(y)}{\kappa}-\frac{F(V)}{\kappa}\Big)dy+O(\eps)\Big]\notag\\
 &\mbox{i.e. }\overline{lim}_{\eps\goto 0} c_\eps \le max_{t\in [0,t_0]}\Big[I_1(t)+O(\eps)\Big]\label{a6.1}
\end{align}
 From the Pohozaev identity we get that $t=1$ is the unique maximum point of the RHS. of (\ref{a6.1}). and hence we have  
 \begin{align}
  \overline{lim}_{\eps\goto 0} \eps^{-(2N+1)}c_\eps \le \Gamma(V)\label{a6.2}
 \end{align}

\begin{lemma}\label{L6.1}
 For a mountain pass solution $u_\eps$ any local maxima $P_\eps$ of $u_\eps$ in $\M$ converges to the boundary of $\M$ as $\eps \goto 0$.
\end{lemma}
\begin{proof}
 First note that (\ref{E2.1}) has constant solutions 0 and 1. with $\Gamma_\eps(0) = 0$ and $\Gamma_\eps(1) = (\frac{1}{2}-\frac{1}{p+1})
 \int_{\M}\frac{dv_g}{|s(q)|^\eta} > 0$. But $c_\eps=O(\eps^{2N+1})$. So $u_\eps$ is not constant. Clearly $u_\eps(P_\eps)\ge 1$. We claim 
 that there exists a constant $C$ such that 
 \begin{align}
  lim_{\eps \goto 0}\frac{1}{\eps} dist_g(P_\eps,\pl\M) < C\label{a6.3}
 \end{align}
If not, let $\exists \ \eps_m \goto 0$ such that $\frac{1}{\eps_m} dist_g(P_{\eps_m},\pl\M)\ge 2m$. Consider a normal co-ordinate $x_{P_{\eps_m}}
 : B_g(P_{\eps_m},m\eps_m)\rightarrow \R^{2N+1}$ and define $w_{\eps_m}(x)=u_{\eps_m}(x_{P_{\eps_m}}^{-1}(\eps_m x))$ in $B(0,m)$. Then we have
 \begin{align}
  \frac{1}{\sqrt{|g|(\eps_m x)}}\sum\pl_i\Big(g^{ik}(\eps_m x)\sqrt{|g|(\eps_m x)}\pl_k w_{\eps_m}\Big)+\frac{w_{\eps_m}}{|\kappa_m(x)|^\eta}
  -\frac{f(w_{\eps_m})}{|\kappa_m(x)|^\eta} = 0 \mbox{ in } B(0,m)\notag
 \end{align}
 Where $\kappa_m(x)=s(x_{P_{\eps_m}}^{-1}(\eps_m x))$ in $B(0,m)$. Let $P_{\eps_m}\goto \tilde{P}$ up-to a subsequence (using compactness argument), 
 and take $\tilde{\kappa}=|s(\tilde{P})|^\eta$. Then by standard elliptic estimate and Sobolev embedding we have $\{w_{\eps_m}\}$
 bounded in $C^{2,\theta}(B(0,m))$ for some $0<\theta<1$ and up-to a subsequence $w_{\eps_m}\rightarrow w$ in $C^2_{loc}(\R^{2N+1})$ where
 $w$ satisfies 
 \be  -\De w + \frac{w}{\tilde \kappa}- \frac{f(w)}{\tilde \kappa}=0, \ w>0 \mbox{ in } \R^{2N+1}\label{a6.4}\ee
  let
 \be  \tilde{J}(w)=\int_{\R^{2N+1}_+}\Big(\frac{1}{2}|\grad w|^2 + \frac{w^2}{2\tilde \kappa}-\frac{F(w)}{\tilde \kappa}\Big)dx\label{a6.5}\ee
 
Define $w_1(y)=w(\sqrt{\frac{\tilde \kappa}{\kappa}}y)$. Then $w_1$ satisfies
 \be  -\De w_1 + \frac{w_1}{\kappa}- \frac{f(w_1)}{\kappa}=0, \ w_1>0 \mbox{ in } \R^{2N+1}\label{a6.6} \ee
Now using the change of variable $x=\sqrt{\frac{\tilde \kappa}{\kappa}}y$ we have
 \begin{align}
  \tilde{J}(w)&=\int_{\R^{2N+1}_+}\Big(\frac{1}{2}|\grad w|^2 + \frac{w^2}{2\tilde \kappa}-\frac{F(w)}{\tilde \kappa}\Big)dx\notag\\
  &=\int_{\R^{2N+1}_+}\Big(\frac{1}{2}\frac{\kappa}{\tilde\kappa}|\grad w_1|^2+ \frac{w_1^2}{2\tilde \kappa}-\frac{F(w_1)}{\tilde \kappa}\Big)\Big(\frac{\tilde\kappa}{\kappa}\Big)^{\frac{2N+1}{2}}dx\notag\\
  &=\Big(\frac{\tilde\kappa}{\kappa}\Big)^{\frac{2N-1}{2}}\int_{\R^{2N+1}_+}\Big(\frac{1}{2}|\grad w_1|^2+ \frac{w_1^2}{2\kappa}-\frac{F(w_1)}{\kappa}\Big)dx\notag
 \end{align}

 Case I : Let $\alpha\le\frac{2}{2N-1}$, then we have $\eta>0$ and take $P_0$ on the inner boundary. The $\frac{\tilde{\kappa}}{\kappa} \ge 1$. Take 
 $V=w_1|_{\R^{2N+1}_+}$. Clearly $V$ satisfies (\ref{E2.2}) with $\mu=(0,...,0,1)$ and $\kappa=|s(P_0)|^\eta$. Then by symmetry we get
 \begin{align}
  \tilde{J}(w)=2\Big(\frac{\tilde\kappa}{\kappa}\Big)^{\frac{2N-1}{2}}\Gamma(V)\notag
 \end{align}
 Now $\eps^{-(2N+1)}c_\eps=\eps^{-(2N+1)}\Gamma_\eps(u_\eps)\rightarrow \tilde{J}(w)$ as $\eps\rightarrow 0$. Then we have the following 
 contradictory argument
 \begin{align}
  lim_{\eps_m\goto 0}\eps_m^{-(2N+1)}c_{\eps_m} = \tilde{J}(w) \ge 2\Gamma(V) \ge 2\overline {lim}_{\eps\goto 0}\eps^{-(2N+1)}c_{\eps}\notag
 \end{align}
 
 Case II : Let $\alpha>\frac{2}{2N-1}$, then we have $\eta>0$ and take $P_0$ on the outer boundary. Then similarly as above we shall arise at a 
 contradiction.

 \end{proof}
  
  To prove the next lemma we need the following result from Ni and Takagi
  \begin{lemma}\label{L6.3}
  Let $\phi\in C^2(B_a)$ be radial function satisfying $\phi'(0)=0$ and $\phi''(0)<0$ in $[0,a]$ the $exists$ a $\delta>0$ such that, if 
  $\psi\in C^2(\overline{B(0,a)})$ satisfies (i) $\grad \psi(0)=0$ and (ii) $||\psi-\phi||_{C^2(\overline{B(0,a)})}<\delta$ then 
  $\grad\psi\neq 0$ for $x\neq 0$
 \end{lemma}
  
 \begin{lemma}\label{L6.2}
  $P_\eps\in\pl\M$ for $\eps$ small enough
 \end{lemma}
 \begin{proof}
  Let $\eps_k \downarrow 0$ be a decreasing sequence such that $P_k :=P_{\eps_k}\in\M$. From lemma (\ref{L6.1}) we have $P_k\goto\tilde{P}$(say)
  $\in\pl\M$(up-to a subsequence). Take the Fermi co-ordinate $x_{\tilde{P}}^\pl$ on a neighborhood of $\tilde{P}$ and let $(x_{\tilde{P}}^\pl)^{-1}$
  is defined on a set containing the closed half-ball $\overline{B_+(0,2\eta)}, \ \eta>0$ and $q_k:=(x_{\tilde{P}}^\pl)^{-1}(P_k)\in B_+(0,\eta)$ 
  for all $k$. Let $v_k(y):=u_{\eps_k}((x_{\tilde{P}}^\pl)^{-1}(y))$ for $y\in \overline{B_+(0,2\eta)}$. Extend $v_k$ to all of $\overline{B(0,2\eta)}$
  by
  \[\tilde{v}_k(y) = 
     \begin{cases}
      v_k(y),&\text{if } y \in \overline{B_+(0,2\eta)}\notag \\
      v_k(y'-y_{2N+1}),&\text{if } y \in \overline{B_-(0,2\eta)}\notag\\
      \end{cases}
      \]
  Define $w_k(z)=\tilde{v}_k(q_k+\eps_kz)$ for $z\in \overline{B(0,\frac{\eta}{\eps})}$. Let $q_k=(q'_k, \theta_k\eps_l), \ q'_k\in \R^{2N}$ and 
  $\theta_k>0$. Then from (\ref{a6.3}) $\theta_k$ is bounded. Then it can be easily shown that $w_k\goto W$ in $C^2_{loc}$ where $W$ satisfies
  \be   -\De W+\frac{W}{|s(\tilde{P})|^\eta}-\frac{f(W)}{|s(\tilde{P})|^\eta}=0 \mbox{ in } \R^{2N+1} \label{a6.15}\ee
 Let $R>0$ be sufficiently large and define $\eps_R=C_0exp(-R/2)$. Then $\exists$ $k_R$ such that
 \begin{align}
  ||w_k-W||_{C^2(\overline{B_{4R}})}\le\eps_R \mbox{ for } k\ge k_R\label{a6.7}
 \end{align}
 We choose $R>\eta_k, \forall k$. 
 
 Now choose $c,d (0<c<d)$ such that $W''(r)<0$ for $r\in [0,c]$ (as $W''(0)<0$) and $W(d)<1$. Since $W'<0$ for $r>0$ one sees $C_*:=min\{|W'(r)|
 |r\in [c,d]\}>0$. If $c\le |z|\le d$, the from (\ref{a6.7}) we have 
 \begin{align}
  |\grad w_k(z)|\ge |\grad W(z)|-|\grad w_k(z)-\grad W(z)|\ge C_*-\eps_R>0\notag
 \end{align}
 provided $C_*>\eps_R$. Apply lemma (\ref{L6.3}) in the ball $\overline{B_b}$ we conclude that $z=0$ is only maximum point of $w_k$ in $B_b$. If 
 $z_k$ is a local maximum point of $w_k$ in $B_R$ then $w_k(z_k)\ge \overline{u}\equiv 1$. choosing $R$ so large that $\eps_R<1-w(b)$ one has if 
 $|z|>b$ then $w_k(z)=W(z)+\eps_R<\overline{u}\equiv 1$. Hence $z_k\in B_b$. Consequently $z_k=0$.
 Now if $\theta_k>0, \ \forall k$ then by definition of $\tilde{v_k} \ q_k^*=(q'_k,-\theta_k\eps_k)$ is also a local maximum of $\tilde{v_k}$ and 
 hence $(0,-\theta_k)$ is another local maximum point of $w_k$ in $B_R$, which is contradictory. So $\theta_k=0$ for $k$ large enough. $\Box$
 \end{proof}
 
 Now let $P_k\in\pl\M$ be a local max of $u_\eps$. Take the Fermi co-ordinate $x_{P_k}^\pl$ around $P_k$ such that $x_{P_k}^\pl$ maps 
 $B_g(P_k,2\delta)$ onto a half ball $B_+(0,2\delta)\cup (\pl\R^{2N+1}\cap B(0,2\delta)) =B$ (say) diffeomorphically. Define $\phi \in C^\infty(\R)$
 such that $\phi(r)=0$ if $r>2\delta$, 1 if $r<\delta$ and define
 \begin{align}
  v_{\eps_k}(x)=\phi(|x|)u_{\eps_k}\Big((x_{P_k}^\pl)^{-1}(\eps_k x)\Big) \mbox{ in } \frac{1}{\eps_k}B\label{a6.8}
 \end{align}
 It can be easily shown that $v_{\eps_k}\goto V$ in $C^2_{loc}(\R^{2N+1}_+)$ with $V\in(\R^{2N+1}_+)$. Clearly $V\ge 0$ and $V$ satisfies 
 (\ref{E2.3}) with $P_0$ lies on the same boundary where $P_k$ lies and $\mu=(0,...0,1)$. Also we have for $\eps_R:=C_0e^{-R}$ with some constant
 $C_0, \ \exists$ integer $k_R$ such that for $k>k_R$ we have 
 \begin{align}
  ||v_{\eps_k}-V||_{C^2(\overline{B_+(0,2R)})}\le\eps_R \mbox{ for } k\ge k_R\label{a6.9}
 \end{align}
 \begin{lemma}\label{L6.4}
  If $u_\eps$ attains a local maxima at $x_0\in \M$ the $\exists$ a positive constant $\eta_0$ independent of $x_0$ and $\eps$ such that 
  $u_\eps(x)\ge \eta_0$ for $x\in B_g(x_0,\eps)\cup \M$ provided $\eps$ sufficiently small
 \end{lemma}
 \begin{proof}
  Easily obtained by Harnack Inequality.(See Ni, Takagi (\cite{MR1219814})).
 \end{proof}
 \begin{lemma}\label{L6.5}
  $u_\eps$ can have only one local maxima.
 \end{lemma}
 \begin{proof}
  If possible let us consider that there is a decreasing sequence $\eps_m\downarrow 0$ such that $u_{\eps_m}$ has two local maximas say $P_1$ and
  $P_2$. From the previous lemmas we have both $P_1$ and $P_2$ are on $\pl\M$ for large $m$. Also as the scaled function $\tilde{v}$ constructed in 
  lemma (\ref{L6.2}), can't have two local maximas in $B_R$ for any $R>0$, we see that $\frac{1}{\eps}dist_g(P_1,P_2)\goto \infty$ as $m\goto\infty$.
  Take the co-ordinate $x_{P_{\eps_m}}^\pl$ around $P_{\eps_m}$ and define $v_{P_{\eps_m}}$ as in (\ref{a6.8}).
  
  In the next step we shall give a lower estimate of the energy functional in order to prove the lemma. Note that
  \begin{align}
   c_{\eps_m}&=\int_{\M}\Big(\frac{u_{\eps_m}f(u_{\eps_m})}{2|s(q)|^\eta}-\frac{F(u_{\eps_m})}{|s(q)|^\eta}\Big)dv_g\notag\\
   &=\Big(\frac{1}{2}-\frac{1}{p+1}\Big)\int_{\M}\frac{u_{\eps_m}^{p+1}}{|s(q)|^\eta}dv_g\notag\\
   &=\Big(\frac{1}{2}-\frac{1}{p+1}\Big)\Big(\int_{B_1}\frac{u_{\eps_m}^{p+1}}{|s(q)|^\eta}dv_g+\int_{\M\backslash B_1}\frac{u_{\eps_m}^{p+1}}{|s(q)|^\eta}dv_g\notag\Big)
  \end{align}
  where $B_1=B_g(P_m,R\eps_m)$ and $m$ is so large that $R\eps_m<\delta$. Now
  \begin{align}
   &\int_{B_1}\frac{u_{\eps_m}^{p+1}}{|s(q)|^\eta}dv_g\notag\\
   =&\int_{B_+(0,R\eps_m)}\frac{v_{\eps_m}^{p+1}}{|s(x_{P_{\eps_m}}^\pl)^{-1}(x)|^\eta}\sqrt{|g|(x)}dx\notag\\
   =&\eps_m^{2N+1}\int_{B_+(0,R)}\frac{v_{\eps_m}^{p+1}(\eps_my)}{|s(x_{P_{\eps_m}}^\pl)^{-1}(\eps_my)|^\eta}\sqrt{|g|(\eps_my)}dx\notag\\
   =&\eps_m^{2N+1}\int_{B_+(0,R)}v_{\eps_m}^{p+1}(\eps_my)\Big(\frac{1}{|s(P_{\eps_m})|^\eta}+O(|\eps_my|)\Big)dy\notag
  \end{align}
  using the change of variable $x=\eps_my$ and using (\ref{a4.5}) and (\ref{a4.8}). Note that $v_{\eps_m}^{p+1}\ge V^{p+1}-|v_{\eps_m}^{p+1}-
  V^{p+1}|$ and we have
  \begin{align}
   &\int_{B_1}\frac{u_{\eps_m}^{p+1}}{|s(q)|^\eta}dv_g\notag\\
   \ge &\eps_m^{2N+1}\int_{B_+(0,R)}\Big(\frac{V^{p+1}}{|s(P_{\eps_m}|^\eta}-\frac{|v_{\eps_m}^{p+1}-V^{p+1}|}{|s(P_{\eps_m})|^\eta}\Big)dx
   +O(|v_{\eps_m}^{p+1}-V^{p+1}|)R^{2N+1}\notag\\
   \ge &\eps_m^{2N+1}\int_{B_+(0,R)}\frac{V^{p+1}}{|s(P_{\eps_m})|^\eta} dy-C_1\eps_RR^{2N+1}\notag\\
   \ge &\eps_m^{2N+1}\int_{\R_+^{2N+1}}\frac{V^{p+1}}{|s(P_{\eps_m})|^\eta} dy-C_1\eps_RR^{2N+1}-C_2\eps\notag
  \end{align}
  Using decay estimate of $V$. On the other hand from lemma (\ref{L6.4}) we have 
  \begin{align}
   \int_{\M\backslash B_1}\frac{u_{\eps_m}^{p+1}}{|s(q)|^\eta}dv_g \ge \int_{B_g(P'_{\eps_m},\eps_m)}\frac{u_{\eps_m}^{p+1}}{|s(q)|^\eta}dv_g
   \ge \eta_0 \int_{B_g(P'_{\eps_m},\eps_m)}dv_g=C_0\eps^{2N+1}\notag
  \end{align}
  So finally we have the lower estimate as
 \begin{align}
  c_{\eps_m}\ge\eps_m^{2N+1}\Big(\Gamma(V)+C_0-C_1 r^{2N}e^{-R}+C_2\eps_m\Big)\label{a6.10}
 \end{align}
  Now define $Z_{\eps,t}^\gamma$ as in section 4 by taking $P_0=P_{\eps_m}$ and $W$ to be a least energy solution of (\ref{E2.2}) with 
  $\mu=(0,...,0,1)$. then from (\ref{a6.2}) we get 
  \begin{align}
  \overline{lim}_{\eps\goto 0} \eps^{-(2N+1)}c_\eps \le \Gamma(W)\label{a6.11}
  \end{align}
  Now $\Gamma(V)\ge\Gamma(W)$ implies (\ref{a6.10}) and (\ref{a6.10}) are contradictory. Hence $u_\eps$ can have only one maximum point which lies
  on the boundary point of $\M$ for $\eps$ small enough.
  \end{proof}
  
  Now using the estimate (\ref{a6.9}), the decay estimate of $V$ in (\ref{a6.9}) and standard comparison principle we have 
  \begin{align}
   u_\eps(x)+|\grad u_\eps(x)|\le C\exp(-\frac{c}{\eps} dist_g(c,P_\eps))\label{a6.12}
  \end{align}
  for some constants $C,c>0$.
  \begin{prop}
   For $\eps$ small enough the following holds \\
   (i) for $\eta>0$, i.e. $\alpha<\frac{2}{2N-1}$, we have $s(P_\eps)=(\frac{2N}{2N-1})^{\frac{1}{2N-1}}a^{\frac{2N}{2N-1}}$\\
   (ii) for $\eta\le 0$, i.e. $\alpha\ge\frac{2}{2N-1}$, we have $s(P_\eps)=(\frac{2N}{2N-1})^{\frac{1}{2N-1}}b^{\frac{2N}{2N-1}}$
  \end{prop}
  \begin{proof}
  We know that for $\eps$ small enough $P_\eps$ is on $\pl\M$.
  
  Case 0 : For $\eta=0$ we have the result of Byeon and Park, \cite{MR2180862}. So we have that a maximum point converges to a point of $\pl\M$ which have maximum 
  mean curvature. From (\ref{a4.3}) we have $s(P_\eps)=(\frac{2N}{2N-1})^{\frac{1}{2N-1}}b^{\frac{2N}{2N-1}}$.
  
  Case I : Let $\eta>0$. Let $V$ be a least energy solution of (\ref{E2.2}) with $\mu=(0,...,0,1)$. Define $Z_{\eps,t}^\gamma$ as in section 4 with
  $P_0$ on the inner boundary and taking the co-ordinate around $P_0$. Then using (\ref{a5.7}) and the Pohozaev identity we get 
  \begin{align}
   \eps^{-(2N+1)}\Gamma_\eps(Z_{\eps,t}^\gamma)&\le \int_{\R^{2N+1}_+}\frac{1}{2}|\grad V|^2dy+\int_{\R^{2N+1}_+}\Big(\frac{V^2(y)}{2\kappa}-\frac{F(V)}{\kappa}\Big)dy + C\eps\notag\\
   &=(\kappa)^{\frac{2N-1}{2}}I(U) + C\eps\label{a6.13} 
  \end{align}
  where $\kappa=(\frac{2N}{2N-1})^{\frac{1}{2N-1}}a^{\frac{2N}{2N-1}}$, $I(U)=\int_{\R^{2N+1}_+}(\frac{1}{2}|\grad U|^2)+\frac{U^2(y)}{2}-F(U))dy$ 
  and $U$ is a least energy solution of (\ref{E2.3}). Now if possible let 
  $s(P_\eps)=(\frac{2N}{2N-1})^{\frac{1}{2N-1}}b^{\frac{2N}{2N-1}}$. Then from (\ref{a6.10}) we have
  \begin{align}
   \eps^{-(2N+1)}c_{\eps_m}\ge\Gamma(W)-C_1 r^{2N}e^{-R}+C_2\eps_m\notag
  \end{align}
  
  where $W$ solves (\ref{E2.2}) with $s(P_0)$ replaced by $\tilde{\kappa}:=s(P_\eps)$. Then as before we can easily show
  \begin{align}
   \eps^{-(2N+1)}c_{\eps_m}\ge (\tilde{\kappa})^{\frac{2N-1}{2}}I(\tilde{U})-C_1 r^{2N}e^{-R}+C_2\eps_m\label{a6.14} 
  \end{align}
  Where $\tilde{U}$ is a solution of (\ref{E2.3}). Note that for $\eta>0$ $\tilde{\kappa}>\kappa$. Also noting that  $U$ is a least energy 
  solution of (\ref{E2.3}) we reach two contradictory inequalities (\ref{a6.13}) and (\ref{a6.14}) when $R$ is large enough.
  
  Case II : Follows similarly as above by taking $P_0$ on the outer boundary for the test function. 
 
  \end{proof}
  \textit{Proof of theorem (\ref{Th0.1})}:
  \begin{proof}
   It is clear that one point concentrating solutions of equation (\ref{E2.1}) can be lifted to $S^1$ concentrating solutions of (\ref{E0.1}) with
   the required properties.
  \end{proof}


\section{Appendices}
 
\subsection{Appendix A : Hopf fibration, Fubini Study metric and the Warped product.}
Here we shall discuss some well known facts about hops fibration on $S^{2N+1}$. All the details can be found in [\cite{MR2243772}].
The $2N+1$ sphere $S^{2N+1}$ can be represented as $S^{2N+1}=I\times(S^{2N-1}\times S^1)$. Also the metric $g$ in (\ref{0.1}) can be represented by
another representation (doubly warped product metric) as
\begin{align}\label{Ap A.1}
 g=dt_1^2+sin^2t_1dS_{2N-1}^2+cos^2t_1dS_1^2
\end{align}
The unit circle acts on both the spheres by complex scalar multiplication as, for $\lambda\in S^1$ and $(z,w)\in S^{2N-1}\times S^1$ we have 
$\lambda\cdotp(z,w)=(\lambda z, \lambda w)$, which induces a fixed point free isometric action on the space. The quotient map
\begin{align}
 I\times(S^{2N-1}\times S^1)\longrightarrow I\times((S^{2N-1}\times S^1)/S^1)\label{Ap A.2}
\end{align}
can be made into Riemann submersion by choosing an appropriate metric on the quotient space. To find this metric we split the canonical metric 
\begin{align}
 dS_{2N-1}=h+g\notag
\end{align}
where $h$ corresponds to the metric along the Hopf fiber and $g$ is the orthogonal complement. Then we got the generalized Hopf fibration 
$S^{2N+1}\rightarrow \cpn$, defined by 
\begin{align}
 (0, \frac{\pi}{2})\times(S^{2N-1}\times S^1) \longrightarrow (0, \frac{\pi}{2})\times((S^{2N-1}\times S^1)/S^1)\label{Ap A.6}
\end{align}
as a Riemann submersion and the corresponding metric is given by 
\begin{align}\label{Ap A.3}
 g_{\cpn}=dt_1^2+sin^2t_1(g+cos^2t_1h)
\end{align}

Now let us take the manifolds $I'$ and $\cpn$ with the metrics $ds^2$ and $g_\cpn$ respectively where $I'$ is the interval 
$(\frac{2N}{2N-1})^{\frac{1}{2N-1}}(a^{\frac{2N}{2N-1}}, b^{\frac{2N}{2N-1}})$. Consider the product manifold $\M=I'\times_f\cpn$.
with the warping function $f=\frac{2N}{2N-1}s$. The warped product metric $g$ is of the form
\begin{align}
 g&=s^*(ds^2)+(f\circ r)^2\sigma^*(g_{\cpn})\notag\\
 &=ds^2+(\frac{2N}{2N-1})^2s^2g_{\cpn}\label{Ap A.4}
\end{align}
Where $s$ and $\sigma$ are projections from $\M$ onto $I'$ and $\cpn$ respectively. For $v \in H^1(\M)$ we have 
\begin{align}
 |\grad_gv|^2=|v_s|^2 + (\frac{2N}{2N-1}s)^{-2}|\grad_{\cpn}v|^2\label{Ap A.5}
\end{align}

 \subsection{Proof of lemma \ref{L2.1}}
 \begin{proof}
  If possible let $z^0=z(r^0,t_1^0,...,t_n^0,\theta_1^0,...,\theta_{N+1}^0)$ be a fixed point of $T_\tau$. So we have 
  \begin{align}
   z(r^0,t_1^0,...,t_n^0,\theta_1^0,...,\theta_{N_1}^0)=z(r^0,t_1^0,...,t_n^0,\theta_1^0+\tau,...,\theta_{N+1}^0+\tau)\notag
  \end{align}
  for all $\tau \in [0,2\pi)$. Now equating the last two coordinates we get $cos(t_1)=0$. So $sin(t_1)\neq 0$. Now compare $x_{2N}$ and use the fact 
  that $sin(t_1)\neq 0$ and we get $cos(t_2)= 0$. In the same way we shall get $cos(t_j)= 0$ for all $j=1,2,...,n$. So $sin(t_j)\neq 0$ for all 
  $j=1,2,...,n$. Now comparing the first coordinates we arrived at contradiction. So $T_\tau$ is a fixed point free group action. 
 \end{proof}

 \subsection{Proof of lemma \ref{L2.2}}
 \begin{proof}
  Let $z=z(r,t_1,...,t_n,\theta_1,...,\theta_{N+1})$ and $z=z(r,t_1,...,t_n,\theta_1^{'},...,\theta_{N+1}^{'})$ such that $\theta_i-\theta_{N+1}
  =\theta_i^{'}-\theta_{N+1}^{'}=\sigma_i$. We need to show that $u(z)=u(z^{'})$. Let $\theta=\theta_1^{'}-\theta_1$. Then we have $\theta_1^{'}=
  \theta_1+\theta$ for some $\theta\in [0,2\pi)$. Hence $\theta_{N+1}^{'}=\theta_1^{'}-\theta_1+\theta_{N+1}=\theta+\theta_{N+1}$ and we get
  $\theta_{N+1}^{'}-\theta_{N+1}=\theta$. Similarly we can show $\theta_{i}^{'}-\theta_{i}=\theta_{N+1}^{'}-\theta_{N+1}=\theta$. So $v$ is well 
  defined fro $u\in H_\sharp^1(A)$.
 \end{proof}
 
 \subsection{Appendix B : The reduction.}
 In this polar co-ordinate the energy functional $J_\eps(u)$ of (\ref{E0.1}) takes the form
 \begin{align}
  J_\eps(u)&=\int_{A}(\frac{\eps^2}{2}|\grad u|^2 + |x|^\alpha\frac{u^2}{2}-|x|^\alpha F(u))dx\notag\\
  &=\int_{I\times S^{2N+1}}[\frac{\eps^2}{2}(u_r^2+\frac{1}{r^2}|\grad_{S^{2N+1}}u|^2)+r^\alpha\frac{u^2}{2}-r^\alpha F(u)]r^{2N+1}drd\sigma_{S^{2N+1}}\notag
 \end{align}
 where I=[a,b]. First we shall do the reduction on the first part of the above integral as
 \begin{align}
  &\int_{I\times S^{2N+1}}(u_r^2+\frac{1}{r^2}|\grad_{S^{2N+1}}u|^2)r^{2N+1}drd\sigma_{S^{2N+1}}\notag\\
  =&\int_I\{\int_{S^{2N+1}}(u_r^2+\frac{1}{r^2}|\grad_{S^{2N+1}}u|^2)\sigma_{S^{2N+1}}\}r^{2N+1}dr\notag
 \end{align}
Now use the change of variable $\phi(r,t_1,...,t_n,\theta_1,...,\theta_{N+1})=(r,t_1,...,t_n,\psi_1,...,\psi_{N})$, as defined above and we get
\begin{align}
 &\int_{S^{2N+1}}(u_r^2+\frac{1}{r^2}|\grad_{S^{2N+1}}u|^2)d\sigma_{S^{2N+1}}\notag\\
 =&2\pi\int_{\cpn}(v_r^2+\frac{1}{r^2}|\grad_{\cpn}v|^2)dV_{\cpn}\notag
\end{align}

So we have
\begin{align}
 &\int_{A}|\grad u|^2dx\notag\\
 =&2\pi\int_I\{\int_{\cpn}(v_r^2+\frac{1}{r^2}|\grad_{\cpn}v|^2)dV_{\cpn}\}r^{2N+1}dr\label{Ap C.1}
\end{align}

let $r=(\frac{2N-1}{2N})^{\frac{1}{2N}}s^{\frac{2N-1}{2N}}$. Then 
\begin{align}
 dr&=(\frac{2N-1}{2N})^{\frac{2N+1}{2N}}s^{\frac{-1}{2N}}\notag\\
 r^{2N+1}dr&=(\frac{2N-1}{2N})^{\frac{2N+11}{N}}s^{2N-\frac{1}{N}}ds\notag
\end{align}
Also $|v_r|^2=(\frac{2N}{2N-1})^{\frac{2N+1}{N}}s^{\frac{1}{N}}|v_s|^2$ and $\frac{1}{r^2}
=(\frac{2N}{2N-1})^{\frac{1}{N}}s^{-2}s^{\frac{1}{N}}$. Then finally we get
\begin{align}
 &\int_{A}|\grad u|^2dx\notag\\
 =&2\pi\int_{I'}\{\int_{\cpn}(v_s^2+(\frac{2N}{2N-1}s)^{-2}|\grad_{\cpn}v|^2)dV_{\cpn}\}s^{2N}ds\label{Ap C.2}
\end{align}
 where $I'=(\frac{2N}{2N-1})^{\frac{1}{2N-1}}(a^{\frac{2N}{2N-1}}, b^{\frac{2N}{2N-1}})$. Also using the same change of variables we get
 \begin{align}
  &\int_{A}(|x|^\alpha\frac{u^2}{2}-|x|^\alpha F(u))dx\notag\\
  =&\int_I\{\int_{S^{2N+1}}(\frac{u^2}{2}-F(u))d\sigma_{S^{2N+1}}\}r^{2N+\alpha+1}dr\notag\\
  =&2\pi\int_{I}\{\int_{\cpn}(\frac{v^2}{2}-F(v))dV_{\cpn}\}r^{2N+\alpha+1}dr\notag\\
  =&2\pi\int_{I'\times\cpn}(\frac{2N-1}{2N})^{\frac{4N+2+\alpha}{2N}}s^{\frac{2N\alpha-2-\alpha}{2N}}(\frac{v^2}{2}-F(v))s^{2N}dsdV_{\cpn}\}\label{Ap C.3}
 \end{align}
 So finally

 \begin{align}
  J_\eps(u)&=2\pi\int_{I'\times\cpn}[\frac{\eps^2}{2}(v_s^2+(\frac{2N}{2N-1}s)^{-2}|\grad_{\cpn}v|^2)\notag\\
  &+(\frac{2N-1}{2N})^{\frac{4N+2+\alpha}{2N}}(\frac{v^2}{2s^{\frac{\alpha+2-2N\alpha}{2N}}}-\frac{F(v)}{s^{\frac{\alpha+2-2N\alpha}{2N}}})]
  s^{2N}dsdV_{\cpn}\label{Ap C.4}
 \end{align}

\bibliographystyle{plain}

\end{document}